\documentclass[12pt]{article}
\usepackage[utf8]{inputenc}
\usepackage[T1]{fontenc}
\usepackage{amsmath,amssymb,amsthm}
\usepackage{geometry}
\usepackage{hyperref}
\usepackage{indentfirst}
\usepackage{booktabs}
\usepackage{xcolor}
\usepackage{fvextra}
\DefineVerbatimEnvironment{code}{Verbatim}{
  breaklines=true,
  breakanywhere=false,
  fontsize=\footnotesize,
  xleftmargin=1em,
}
\DefineVerbatimEnvironment{term}{Verbatim}{
  breaklines=true,
  breakanywhere=true,
  fontsize=\footnotesize,
  xleftmargin=1em,
}

\newtheorem{theorem}{Theorem}
\newtheorem{lemma}{Lemma}
\newtheorem{proposition}{Proposition}
\newtheorem{corollary}{Corollary}
\newtheorem{conjecture}{Conjecture}
\theoremstyle{remark}
\newtheorem{remark}{Remark}
\theoremstyle{definition}
\newtheorem{definition}{Definition}

\title{Computational Evidence Against Quadratic--Cubic Factorization for the Second Cuboid Quintic}

\author{Valery Asiryan\\[3pt]
\small \texttt{asiryanvalery@gmail.com}
\and
Randall L. Rathbun\\[3pt]
\small \texttt{randallrathbun@gmail.com}}
\date{\small January 22, 2026}

\begin{document}

\maketitle

\begin{abstract}
Let $Q_{p,q}(t)\in\mathbb{Z}[t]$ be Sharipov's even monic degree-$10$ \emph{second cuboid polynomial} depending on coprime integers $p\neq q>0$.
Writing $Q_{p,q}(t)$ as a quintic in $t^{2}$ produces an associated monic quintic polynomial.
After the weighted normalization $r=p/q$ and $s=r^{2}$ we obtain a one-parameter family $P_s(x)\in\mathbb{Q}[x]$ such that
\[
Q_{p,q}(t)=q^{20}\,P_s\!\left(\frac{t^{2}}{q^{4}}\right)\qquad\text{with}\qquad s=\left(\frac{p}{q}\right)^{2}.
\]

Assuming a quadratic divisor $x^{2}+ax+b$ with $a,b\in\mathbb{Q}$, we reduce divisibility of $P_s(x)$ to the vanishing of an explicit remainder
\[
R(x)=R_{1}(s,a,b)\,x+R_{0}(s,a,b).
\]
A key structural observation is that $R_1$ and $R_0$ are quadratic in $b$ and that, on the equation $R_1=0$, the second condition becomes linear in $b$.
This yields a one-direction elimination to a plane obstruction curve $F(s,a)=0$ with $F\in\mathbb{Z}[s,a]$, without any ``lifting-back'' issues: when the linear coefficient is nonzero, the parameter $b$ is forced to be the rational value $b=C/L$.
We isolate the degenerate locus $L=C=0$ and show it produces only $s=\pm 1$ (hence only $s=1$ in the cuboid domain $s>0$).

Let $\overline{C}\subset\mathbb{P}^{2}$ be the projective closure of $F(s,a)=0$.
Using \textsc{Magma} we perform a height-bounded search for rational points on $\overline{C}$.
With bound $H=10^{9}$, the search returns $8$ rational points, whose affine part has $s\in\{-1,0,1\}$.
In particular, no affine rational point with $s>0$ and $s\neq 1$ is found up to this bound.
This provides strong computational evidence that for rational $s>0$, $s\neq 1$, the quintic $P_s(x)$ admits no quadratic factor over $\mathbb{Q}$ (equivalently, no \emph{$2+3$} (quadratic--cubic) factorization over $\mathbb{Q}$), and yields a conditional exclusion assuming completeness of the rational-point enumeration on $\overline{C}$.


\end{abstract}


\section{Introduction}\label{sec:intro}

The perfect cuboid problem asks for a rectangular box with integer edges such that all three face diagonals and the space diagonal are integers.
In a framework due to R.\,A.\,Sharipov \cite{Sharipov2011Cuboids,Sharipov2011Note}, one is led to explicit parameter-dependent even polynomials whose irreducibility is conjectured for coprime parameters.
In particular, Sharipov defines the \emph{second cuboid polynomial} $Q_{p,q}(t)\in\mathbb{Z}[t]$ of degree $10$ and formulates the conjecture that $Q_{p,q}(t)$ is irreducible over $\mathbb{Z}$ for coprime $p\neq q>0$.

\medskip
\noindent\textbf{Goal.}
We reduce the \emph{quadratic--cubic} ($2+3$) case for the normalized associated quintic $P_s(x)$ to rational points on an explicit plane curve $F(s,a)=0$, and we provide computational evidence (via a height-bounded search) that no such factorization occurs for $s\in\mathbb{Q}_{>0}$ with $s\neq 1$.

\medskip
\noindent\textbf{Method.}
We work directly with the divisibility condition by a generic quadratic $x^{2}+ax+b$.
The Euclidean remainder has the form $R_1x+R_0$; remarkably, $R_1$ and $R_0$ are quadratic in $b$, and on the locus $R_1=0$ the second equation becomes linear in $b$.
This yields an explicit plane obstruction curve $F(s,a)=0$ in the $(s,a)$-plane.
Determining all rational points on curves of genus $>1$ is a subtle Diophantine problem; finiteness is guaranteed by Faltings' theorem \cite{Faltings1983}, while explicit determination typically requires additional methods (e.g.\ Chabauty--Coleman and refinements) \cite{Chabauty1941,Coleman1985,Stoll2006}.
In this note we instead perform a height-bounded search for rational points on the projective closure $\overline{C}$ of $F(s,a)=0$ using \textsc{Magma}~\cite{Magma}.


\section{The second cuboid polynomial and its associated quintic}\label{sec:poly}

\subsection{Sharipov's second cuboid polynomial \texorpdfstring{$Q_{p,q}(t)$}{Qp,q(t)}}
Let $p,q\in\mathbb{Z}_{>0}$ be coprime and $p\neq q$.
The second cuboid polynomial is the even monic degree-$10$ polynomial
\begin{align}
Q_{p,q}(t)
={}&t^{10} + (2 q^{2} + p^{2})(3 q^{2} - 2 p^{2})\, t^{8} \nonumber\\
&+ (q^{8} + 10 p^{2} q^{6} + 4 p^{4} q^{4} - 14 p^{6} q^{2} + p^{8})\, t^{6} \nonumber\\
&- p^{2} q^{2}\,(q^{8} - 14 p^{2} q^{6} + 4 p^{4} q^{4} + 10 p^{6} q^{2} + p^{8})\, t^{4} \nonumber\\
&- p^{6} q^{6}\,(q^{2} + 2 p^{2})(-2 q^{2} + 3 p^{2})\, t^{2}
- p^{10} q^{10}\in\mathbb{Z}[t].\label{eq:Qpq}
\end{align}
This is the polynomial denoted $Q_{p,q}(t)$ in~\cite{Sharipov2011Cuboids,Sharipov2011Note}.

\subsection{Weighted normalization to \texorpdfstring{$Q_r(u)$}{Qr(u)}}
The polynomial \eqref{eq:Qpq} is weighted-homogeneous of total weight $20$ for
\[
\deg(p)=\deg(q)=1,\qquad \deg(t)=2,
\]
hence one may normalize to a one-parameter family.

\begin{lemma}[Normalization]\label{lem:normalize}
Let $q\neq 0$ and set
\[
r:=\frac{p}{q}\in\mathbb{Q},\qquad u:=\frac{t}{q^{2}}\in\mathbb{Q}.
\]
Then
\begin{equation}\label{eq:normalize}
Q_{p,q}(t)=q^{20}\,Q_r(u),
\end{equation}
where
\begin{equation}\label{eq:Qr}
\begin{aligned}
Q_r(u)
={}& u^{10} + (2+r^2)(3-2r^2)\,u^{8} + \bigl(1+10r^2+4r^4-14r^6+r^8\bigr)\,u^{6}\\
& - r^2\bigl(1-14r^2+4r^4+10r^6+r^8\bigr)\,u^{4}
 - r^6(1+2r^2)(-2+3r^2)\,u^{2}
 - r^{10}\in\mathbb{Q}[u].
\end{aligned}
\end{equation}
\end{lemma}

\begin{proof}
Substitute $p=rq$ and $t=q^2u$ into \eqref{eq:Qpq} and factor out $q^{20}$.
\end{proof}

\subsection{The associated quintic \texorpdfstring{$P_s(x)$}{Ps(x)}}
Since $Q_r(u)$ is even, it is a quintic in $x=u^2$.

\begin{definition}[Second cuboid quintic]\label{def:Ps}
Let $s:=r^2\in\mathbb{Q}_{\ge 0}$.
Define $P_s(x)\in\mathbb{Q}[x]$ by the identity
\begin{equation}\label{eq:QrPs}
Q_r(u)=P_s(u^2).
\end{equation}
Equivalently, $P_s(x)$ is the monic quintic
\begin{equation}\label{eq:Ps}
\begin{aligned}
P_s(x)
={}&x^{5} + (2+s)(3-2s)\,x^{4} + (1+10s+4s^{2}-14s^{3}+s^{4})\,x^{3}\\
& - s(1-14s+4s^{2}+10s^{3}+s^{4})\,x^{2}
 - s^{3}(1+2s)(-2+3s)\,x
 - s^{5}.
\end{aligned}
\end{equation}
\end{definition}

\begin{proof}[Derivation]
Substitute $x=u^2$ into \eqref{eq:Qr} and set $s=r^2$.
\end{proof}

\begin{definition}[$2+3$ factorization]\label{def:23}
Let $K$ be a field of characteristic $0$.
We say that a monic quintic $P(x)\in K[x]$ admits a \emph{$2+3$ factorization over $K$} if it is divisible in $K[x]$ by a quadratic polynomial (equivalently, $P(x)=D(x)H(x)$ with $\deg D=2$ and $\deg H=3$).
\end{definition}

\begin{remark}\label{rem:domain}
In the original cuboid setting we have $p,q>0$, hence $r>0$ and $s=r^2\in\mathbb{Q}_{>0}$.
Moreover $p\neq q$ is equivalent to $r\neq 1$, i.e.\ $s\neq 1$.
\end{remark}

\begin{lemma}[Quadratic factors and even quartic factors]\label{lem:quartic}
Let $s=r^{2}$ with $r=p/q\in\mathbb{Q}$ and $q\neq 0$.
If $P_s(x)$ is divisible in $\mathbb{Q}[x]$ by $x^{2}+ax+b$ with $a,b\in\mathbb{Q}$, then $Q_{p,q}(t)$ is divisible in $\mathbb{Q}[t]$ by the even quartic polynomial
\[
t^{4}+a q^{4}t^{2}+b q^{8}.
\]
\end{lemma}

\begin{proof}
If $(x^{2}+ax+b)\mid P_s(x)$ then $(u^{4}+a u^{2}+b)\mid P_s(u^{2})=Q_r(u)$.
Substituting $u=t/q^{2}$ and multiplying by $q^{8}$ yields $(t^{4}+a q^{4}t^{2}+b q^{8})\mid Q_{p,q}(t)$ by \eqref{eq:normalize}.
\end{proof}


\section{Quadratic divisors and explicit remainder equations}\label{sec:remainder}

Fix $s\in\mathbb{Q}$ and consider a generic monic quadratic
\begin{equation}\label{eq:D}
D(x):=x^{2}+ax+b,\qquad a,b\in\mathbb{Q}.
\end{equation}
Write the Euclidean division in $\mathbb{Q}(s,a,b)[x]$ as
\begin{equation}\label{eq:Quotrem}
P_s(x)=Q(x)\,D(x)+R(x),\qquad \deg R<2,
\end{equation}
so that $R(x)=R_{1}(s,a,b)\,x+R_{0}(s,a,b)$.
Then $D\mid P_s$ is equivalent to $R_1=R_0=0$.

\begin{lemma}[Explicit remainder]\label{lem:Rvectors}
Let $P_s(x)$ be as in \eqref{eq:Ps} and let $D(x)=x^{2}+ax+b$.
Then the remainder in \eqref{eq:Quotrem} is
\begin{equation}\label{eq:Rform}
R(x)=\bigl(b^{2}+u(s,a)\,b+v(s,a)\bigr)\,x+\bigl(m(s,a)\,b^{2}+n(s,a)\,b-s^{5}\bigr),
\end{equation}
where $u,v,m,n\in\mathbb{Z}[s,a]$ are given explicitly by
\begin{align}
u(s,a)={}&-3a^{2}+(12-4s^{2}-2s)a+(-s^{4}+14s^{3}-4s^{2}-10s-1),\label{eq:u}\\[2pt]
v(s,a)={}&a^{4}+(2s^{2}+s-6)a^{3}+(s^{4}-14s^{3}+4s^{2}+10s+1)a^{2}\nonumber\\
&\qquad+(s^{5}+10s^{4}+4s^{3}-14s^{2}+s)a+(-6s^{5}+s^{4}+2s^{3}),\label{eq:v}\\[2pt]
m(s,a)={}&-2a-2s^{2}-s+6,\label{eq:m}\\[2pt]
n(s,a)={}&a^{3}+(2s^{2}+s-6)a^{2}+(s^{4}-14s^{3}+4s^{2}+10s+1)a\nonumber\\
&\qquad+(s^{5}+10s^{4}+4s^{3}-14s^{2}+s).\label{eq:n}
\end{align}
\end{lemma}

\begin{proof}
In the quotient ring $\mathbb{Q}[s,a,b][x]/(x^{2}+ax+b)$ we have $x^{2}\equiv-ax-b$, hence every power $x^{k}$ reduces to a linear expression in $x$.
Reducing each term of $P_s(x)$ modulo $D(x)$ yields a remainder of the form \eqref{eq:Rform}, and collecting coefficients gives \eqref{eq:u}--\eqref{eq:n}.
For completeness and reproducibility, Script~\ref{scr:magma} in Appendix~\ref{app:scripts} cross-checks \eqref{eq:Rform} against \textsc{Magma}'s \texttt{Quotrem} computation.
\end{proof}


\section{Linearization in \texorpdfstring{$b$}{b} and an obstruction curve}\label{sec:obstruction}

\subsection{Linearization on the locus \texorpdfstring{$R_1=0$}{R1=0}}

Let
\begin{equation}\label{eq:R1R0}
R_1(s,a,b):=b^{2}+u(s,a)\,b+v(s,a),\qquad
R_0(s,a,b):=m(s,a)\,b^{2}+n(s,a)\,b-s^{5}.
\end{equation}

\begin{lemma}[One-direction reduction]\label{lem:linearize}
Define
\begin{equation}\label{eq:LC}
L(s,a):=n(s,a)-m(s,a)\,u(s,a),\qquad C(s,a):=m(s,a)\,v(s,a)+s^{5}.
\end{equation}
Then, on the equation $R_1(s,a,b)=0$, we have
\begin{equation}\label{eq:R0linear}
R_0(s,a,b)=L(s,a)\,b-C(s,a).
\end{equation}
\end{lemma}

\begin{proof}
If $R_1=0$ then $b^{2}=-u\,b-v$. Substituting this into $R_0=m b^{2}+n b-s^{5}$ gives
\[
R_0=m(-u b-v)+n b-s^{5}=(n-mu)b-(mv+s^{5})=L b-C,
\]
which is \eqref{eq:R0linear}.
\end{proof}

\subsection{The obstruction polynomial \texorpdfstring{$F(s,a)$}{F(s,a)}}

\begin{definition}[Obstruction polynomial]\label{def:Fsa}
Let $u,v,m,n,L,C$ be as in \eqref{eq:u}--\eqref{eq:n} and \eqref{eq:LC}.
Define
\begin{equation}\label{eq:Fsa}
F(s,a):=C(s,a)^{2}+u(s,a)\,L(s,a)\,C(s,a)+v(s,a)\,L(s,a)^{2}\in\mathbb{Z}[s,a].
\end{equation}
\end{definition}

\begin{proposition}[Divisibility criterion]\label{prop:criterion}
Fix $s\in\mathbb{Q}$.
The following are equivalent:
\begin{enumerate}
\item $P_s(x)$ admits a $2+3$ factorization over $\mathbb{Q}$;
\item there exist $a,b\in\mathbb{Q}$ such that $R_1(s,a,b)=R_0(s,a,b)=0$.
\end{enumerate}
Moreover, if such $a,b$ exist then either
\begin{enumerate}
\item[(A)] $L(s,a)\neq 0$ and $F(s,a)=0$, in which case necessarily $b=C(s,a)/L(s,a)$; or
\item[(B)] $L(s,a)=0$ and $C(s,a)=0$ (the \emph{degenerate locus}).
\end{enumerate}
\end{proposition}

\begin{proof}
The equivalence (1)$\Leftrightarrow$(2) is immediate from $D\mid P_s$ $\Leftrightarrow$ remainder $R$ vanishes identically, i.e.\ $R_1=R_0=0$.

Assume $R_1=R_0=0$. By Lemma~\ref{lem:linearize} we have $R_0=L b-C$, hence either $L\neq 0$ and $b=C/L$, or $L=0$ and $C=0$.
If $L\neq 0$, substituting $b=C/L$ into $R_1=b^{2}+u b+v=0$ and clearing denominators gives $F(s,a)=0$ by \eqref{eq:Fsa}.
\end{proof}

\begin{proposition}[Degrees and the special fiber $s=1$]\label{prop:degrees}
The polynomial $F(s,a)\in\mathbb{Z}[s,a]$ satisfies
\[
\deg_{s}F=16,\qquad \deg_{a}F=10.
\]
Moreover,
\[
F(1,a)=-a^{4}(a-2)^{6}.
\]
\end{proposition}

\begin{proof}
This is verified by \textsc{Magma} in Appendix~\ref{app:scripts} (see the transcript in Script~\ref{scr:magma}).
The specialization at $s=1$ reflects the factorization
\[
P_{1}(x)=(x-1)(x+1)^{4},
\]
so that both $x^{2}-1$ (corresponding to $a=0$) and $(x+1)^{2}$ (corresponding to $a=2$) divide $P_1(x)$.
\end{proof}


\section{The degenerate locus \texorpdfstring{$L=C=0$}{L=C=0}}\label{sec:degenerate}

The reduction in Proposition~\ref{prop:criterion} is one-direction and produces the explicit plane curve $F(s,a)=0$ when $L\neq 0$.
The remaining possibility is the degenerate locus $L=C=0$.

\begin{proposition}[Degenerate locus classification]\label{prop:degenerate}
The system
\[
L(s,a)=0,\qquad C(s,a)=0
\]
has the following rational solutions:
\[
(s,a)=(1,2)\quad\text{and}\quad (s,a)=(-1,2).
\]
In particular, for $s\in\mathbb{Q}_{>0}$ with $s\neq 1$ the degenerate locus is empty.
\end{proposition}

\begin{proof}[Computational proof in \textsc{Magma}]
We eliminate $s$ by the resultant $\mathrm{Res}_{s}(L,C)\in\mathbb{Q}[a]$.
\textsc{Magma} computes
\[
\mathrm{Res}_{s}(L,C)=(a-2)^{6}\cdot G(a),
\]
where $G(a)\in\mathbb{Q}[a]$ has degree $21$ and has no rational roots.
Hence any rational solution must satisfy $a=2$.
Substituting $a=2$ and computing $\gcd(L(s,2),C(s,2))$ yields $(s-1)^{2}(s+1)$, so $s=\pm 1$.
The relevant transcript is included in Appendix~\ref{app:scripts} (Script~\ref{scr:magma}, Step~4).
\end{proof}


\section{Rational points on the obstruction curve and the $2+3$ case}\label{sec:points}

Let $F(s,a)$ be as in Definition~\ref{def:Fsa}.
Let $\widehat{F}(S,A,Z)$ be the homogeneous polynomial of total degree $17$ obtained by homogenizing $F(S/Z,A/Z)$ in $\mathbb{Q}[S,A,Z]$, and let
\begin{equation}\label{eq:Cproj}
\overline{C}:\quad \widehat{F}(S,A,Z)=0 \subset \mathbb{P}^{2}_{\mathbb{Q}}
\end{equation}
be the projective closure.

\begin{proposition}[Genus and singularities]\label{prop:genus}
The plane curve $\overline{C}$ has degree $17$, arithmetic genus $120$, and geometric genus $7$.
Its singular rational points in $\mathbb{P}^{2}(\mathbb{Q})$ are
\[
(-1:2:1),\ (0:0:1),\ (1:0:1),\ (1:2:1),\ (-1:1:0),\ (0:1:0).
\]
\end{proposition}

\begin{proof}[Computational proof in \textsc{Magma}]
This follows from \textsc{Magma}'s commands \texttt{Degree}, \texttt{Genus}, \texttt{GeometricGenus}, \texttt{ArithmeticGenus}, and \texttt{SingularPoints} applied to $\overline{C}$.
See Script~\ref{scr:magma} in Appendix~\ref{app:scripts}, Step~5.
\end{proof}

\begin{proposition}[Height-bounded rational points on $\overline{C}$]\label{prop:CPQ}
Let $H=10^{9}$.
The \textsc{Magma} command \texttt{RationalPoints(Cproj : Bound := H)} returns the following $8$ rational points on $\overline{C}$:
\begin{equation}\label{eq:CQ}
\{(1:2:1),\ (0:1:0),\ (0:0:1),\ (0:6:1),\ (1:0:0),\ (-1:2:1),\ (1:0:1),\ (-1:1:0)\}.
\end{equation}
In particular, among these points the affine rational points (with $Z\neq 0$) are exactly
\[
(s,a)\in\{(1,2),\ (1,0),\ (0,0),\ (0,6),\ (-1,2)\},
\]
so no affine rational point with $s>0$ and $s\neq 1$ is found within this search.
\end{proposition}

\begin{proof}[Computational proof in \textsc{Magma}]
This is the output of the computation recorded in Appendix~\ref{app:scripts}, Script~\ref{scr:magma}, Step~5, with search bound $H=10^{9}$.
\end{proof}

\begin{conjecture}[Completeness of the rational-point list]\label{conj:complete}
The points listed in \eqref{eq:CQ} are all rational points on $\overline{C}$, i.e.\ $\overline{C}(\mathbb{Q})$ consists of exactly these $8$ points.
\end{conjecture}

\begin{theorem}[Conditional exclusion of $2+3$ factorization]\label{thm:main}
Assume Conjecture~\ref{conj:complete}.
Let $s\in\mathbb{Q}_{>0}$ with $s\neq 1$.
Then the quintic $P_s(x)$ admits no $2+3$ factorization over $\mathbb{Q}$ (equivalently, it has no quadratic factor over $\mathbb{Q}$).
\end{theorem}

\begin{proof}
Assume, for contradiction, that $P_s$ admits a $2+3$ factorization over $\mathbb{Q}$ for some $s\in\mathbb{Q}_{>0}$ with $s\neq 1$.
Then by Proposition~\ref{prop:criterion} there exist $a,b\in\mathbb{Q}$ with $R_1(s,a,b)=R_0(s,a,b)=0$.

By Proposition~\ref{prop:degenerate}, the degenerate case $L=C=0$ can occur for $s>0$ only when $s=1$, which is excluded.
Hence $L(s,a)\neq 0$ and Proposition~\ref{prop:criterion}(A) applies, giving $F(s,a)=0$.
Thus $(s,a)$ is an affine rational point on $\overline{C}$.

By Conjecture~\ref{conj:complete}, every affine rational point on $\overline{C}$ has $s\in\{-1,0,1\}$.
Since $s>0$ and $s\neq 1$, no such point exists.
This contradiction shows that $P_s$ has no quadratic factor over $\mathbb{Q}$, hence admits no $2+3$ factorization.
\end{proof}

\begin{corollary}[Conditional exclusion of even quartic factors of $Q_{p,q}(t)$]\label{cor:quartic}
Assume Conjecture~\ref{conj:complete}.
Let $p,q\in\mathbb{Z}_{>0}$ be coprime with $p\neq q$, and set $s=(p/q)^{2}$.
Then $Q_{p,q}(t)$ has no even quartic factor over $\mathbb{Q}$.
\end{corollary}

\begin{proof}
If $Q_{p,q}(t)$ had an even quartic factor, then by Lemma~\ref{lem:quartic} the associated $P_s$ would have a quadratic factor, contradicting Theorem~\ref{thm:main}.
\end{proof}

\begin{remark}[Irreducibility of $Q_{p,q}(t)$]\label{rem:irred}
The reduction above isolates the $2+3$ case within the normalized quintic family $P_s(x)$.
The computational evidence in Proposition~\ref{prop:CPQ} supports the expectation that no such factorization occurs for $s\in\mathbb{Q}_{>0}$ with $s\neq 1$, but a complete determination of $\overline{C}(\mathbb{Q})$ would be required for an unconditional theorem; cf.\ the general context of rational points on higher-genus curves \cite{Faltings1983,Chabauty1941,Coleman1985,Stoll2006}.
\end{remark}


\appendix
\section{\textsc{Magma} script and transcript}\label{app:scripts}

\noindent
All computer-assisted steps in this note are executed in \textsc{Magma}~\cite{Magma}.
Script~\ref{scr:magma} constructs the obstruction polynomial $F(s,a)$, cross-checks it against the Euclidean remainder and the resultant in $b$, classifies the degenerate locus $L=C=0$, and performs a height-bounded search for rational points on $\overline{C}$ via \texttt{RationalPoints(Cproj : Bound := H)} with $H=10^{9}$.

\subsection*{Script \ref{scr:magma}: Obstruction curve and rational points}
\label{scr:magma}

\noindent\textbf{Code.}
\begin{code}
// ---------- Pretty printing ----------
procedure Banner(msg)
    print "\n-------------------------------------------------------------------------------";
    print msg;
    print "-------------------------------------------------------------------------------\n";
end procedure;

// Setup rational field
Q := RationalField();

Banner("Step 1. Define u,v,m,n,L,C,F in Q[s,a] (closed-form, no Resultant)");

// Work ring in (s,a)
R<s,a> := PolynomialRing(Q, 2);

// Explicit polynomials u,v,m,n
u := -3*a^2 + (12 - 4*s^2 - 2*s)*a + (-s^4 + 14*s^3 - 4*s^2 - 10*s - 1);

v := a^4 + (2*s^2 + s - 6)*a^3
     + (s^4 - 14*s^3 + 4*s^2 + 10*s + 1)*a^2
     + (s^5 + 10*s^4 + 4*s^3 - 14*s^2 + s)*a
     + (-6*s^5 + s^4 + 2*s^3);

m := -2*a - 2*s^2 - s + 6;

n := a^3 + (2*s^2 + s - 6)*a^2
     + (s^4 - 14*s^3 + 4*s^2 + 10*s + 1)*a
     + (s^5 + 10*s^4 + 4*s^3 - 14*s^2 + s);

// Linearized condition on R1=0:  L*b - C = 0
L := n - m*u;
C := m*v + s^5;

// Main obstruction polynomial in Q[s,a]
F := C^2 + u*L*C + v*L^2;

print "deg_s F =", Degree(F, 1);
print "deg_a F =", Degree(F, 2);

Banner("Step 2 (cross-check). Compute remainder by division in Q[s,a,b][x] and verify u,v,m,n and F=Res_b(R1,R0)");

// Ring in parameters (s,a,b)
R3<s3,a3,b3> := PolynomialRing(Q, 3);
Px<x> := PolynomialRing(R3);

// Coefficients of Ps(x)
c4 := (2+s3)*(3-2*s3);
c3 := 1 + 10*s3 + 4*s3^2 - 14*s3^3 + s3^4;
c2 := -s3*(1 - 14*s3 + 4*s3^2 + 10*s3^3 + s3^4);
c1 := -s3^3*(1+2*s3)*(-2+3*s3);
c0 := -s3^5;

Ps := x^5 + c4*x^4 + c3*x^3 + c2*x^2 + c1*x + c0;
D  := x^2 + a3*x + b3;

_, Rem := Quotrem(Ps, D);
R1_div := Coefficient(Rem, 1);
R0_div := Coefficient(Rem, 0);

print "deg_b R1 (division) =", Degree(R1_div, 3);
print "deg_b R0 (division) =", Degree(R0_div, 3);

// Coerce u,v,m,n to R3 by substitution
u3 := Evaluate(u, [s3, a3]);
v3 := Evaluate(v, [s3, a3]);
m3 := Evaluate(m, [s3, a3]);
n3 := Evaluate(n, [s3, a3]);

R1_model := b3^2 + u3*b3 + v3;
R0_model := m3*b3^2 + n3*b3 - s3^5;

print "Check R1_div == R1_model ?  ", (R1_div - R1_model) eq 0;
print "Check R0_div == R0_model ?  ", (R0_div - R0_model) eq 0;

// Verify F = Res_b(R1,R0) in Q[s,a]
S<ss,aa> := PolynomialRing(Q, 2);
T<bb> := PolynomialRing(S);

uS := Evaluate(u, [ss, aa]);
vS := Evaluate(v, [ss, aa]);
mS := Evaluate(m, [ss, aa]);
nS := Evaluate(n, [ss, aa]);

R1S := bb^2 + uS*bb + vS;
R0S := mS*bb^2 + nS*bb - ss^5;

Fres := Resultant(R1S, R0S);
FS   := Evaluate(F, [ss, aa]);

print "Check Resultant == F ?      ", (Fres - FS) eq 0;
print "deg_s(Resultant) =", Degree(Fres, 1), " ; deg_a(Resultant) =", Degree(Fres, 2);

// ---------- Fiber tests ----------
Banner("Step 3. Fiber tests: specialize s=s0 and check rational roots in a");

procedure CheckFiber(s0)
    Pa<av> := PolynomialRing(Q);
    F_s0 := Evaluate(F, [s0, av]);

    print "\n--- Fiber s =", s0, "---";
    print "Factorization(F_s0):", Factorization(F_s0);

    rts := Roots(F_s0);
    print "Rational roots in a:", rts;

    if #rts gt 0 then
        print "Derived (a,b) solutions (when L!=0):";
        for rt in rts do
            a0 := rt[1];
            L0 := Evaluate(L, [s0, a0]);
            C0 := Evaluate(C, [s0, a0]);

            if L0 ne 0 then
                b0 := C0/L0;

                u0 := Evaluate(u, [s0, a0]);
                v0 := Evaluate(v, [s0, a0]);
                m0 := Evaluate(m, [s0, a0]);
                n0 := Evaluate(n, [s0, a0]);

                ok1 := (b0^2 + u0*b0 + v0) eq 0;
                ok0 := (m0*b0^2 + n0*b0 - s0^5) eq 0;

                print <a0, b0, ok1, ok0>;
            else
                print <a0, "L=0", "C=0 ?", (C0 eq 0)>;
            end if;
        end for;
    end if;
end procedure;

CheckFiber(Q!1);
CheckFiber(Q!4);
CheckFiber(Q!(4/9));
CheckFiber(Q!2);

// ---------- Degenerate locus L=C=0 ----------
Banner("Step 4. Degenerate locus: solve L(s,a)=0 and C(s,a)=0 via resultant in s");

Qa<aa> := PolynomialRing(Q);
Ts<ss> := PolynomialRing(Qa);

L_ss := Evaluate(L, [ss, aa]);
C_ss := Evaluate(C, [ss, aa]);

ResA := Resultant(L_ss, C_ss);
ResA := PrimitivePart(ResA);

print "Resultant Res_s(L,C) as polynomial in a has degree:", Degree(ResA);
print "Factorization of Res_s(L,C) in Q[a]:";
facResA := Factorization(ResA);
print facResA;

linroots := [];
for fe in facResA do
    f  := fe[1];
    ex := fe[2];
    if Degree(f) eq 1 then
        c1 := Coefficient(f, 1);
        c0 := Coefficient(f, 0);
        r  := -c0/c1;
        Append(~linroots, <r, ex>);
    end if;
end for;

print "Rational a-roots forced by Res_s(L,C)=0 (linear factors only):", linroots;

if #linroots gt 0 then
    for rr in linroots do
        a0 := rr[1];

        Qs<sv> := PolynomialRing(Q);
        L_s := Evaluate(L, [sv, a0]);
        C_s := Evaluate(C, [sv, a0]);

        g := GCD(L_s, C_s);

        print "\n--- Degenerate analysis at a =", a0, "---";
        print "gcd_s(L,C) =", g;
        print "Roots of gcd in s (rational):", Roots(g);
        print "Factorization of L(s,a0):", Factorization(L_s);
        print "Factorization of C(s,a0):", Factorization(C_s);
    end for;
end if;

// ---------- Geometry and rational points ----------
Banner("Step 5. Projective curve defined by F(s,a)=0: singularities and genus");

P2<X,Y,Z> := ProjectiveSpace(Q, 2);
F_XY := Evaluate(F, [X, Y]);
Fh := Homogenization(F_XY, Z);
Cproj := Curve(P2, Fh);

print "Projective curve defined. Degree =", Degree(Cproj);

try
    sing := SingularPoints(Cproj);
    print "Singular points (projective):";
    print sing;
catch e
    print "SingularPoints warning:", e;
end try;

gA := ArithmeticGenus(Cproj);
gG := GeometricGenus(Cproj);
g  := Genus(Cproj);

print "Arithmetic genus =", gA;
print "Geometric genus  =", gG;
print "Genus            =", g;

H := 10^9;
print "Computing rational points on Cproj with bound H =", H;

pts := RationalPoints(Cproj : Bound := H);
print "\nRational Points found on Cproj:";
print pts;

print "\nInterpretation:";
for pt in pts do
    if pt[3] ne 0 then
        s_val := pt[1]/pt[3];
        a_val := pt[2]/pt[3];
        print "Affine solution (s, a) =", <s_val, a_val>;
    else
        print "Point at infinity:", pt;
    end if;
end for;
\end{code}

\noindent\textbf{Transcript.}
\begin{term}
-------------------------------------------------------------------------------
Step 1. Define u,v,m,n,L,C,F in Q[s,a] (closed-form, no Resultant)
-------------------------------------------------------------------------------

deg_s F = 16
deg_a F = 10

-------------------------------------------------------------------------------
Step 2 (cross-check). Compute remainder by division in Q[s,a,b][x] and verify
u,v,m,n and F=Res_b(R1,R0)
-------------------------------------------------------------------------------

deg_b R1 (division) = 2
deg_b R0 (division) = 2
Check R1_div == R1_model ?   true
Check R0_div == R0_model ?   true
Check Resultant == F ?       true
deg_s(Resultant) = 16  ; deg_a(Resultant) = 10

-------------------------------------------------------------------------------
Step 3. Fiber tests: specialize s=s0 and check rational roots in a
-------------------------------------------------------------------------------


--- Fiber s = 1 ---
Factorization(F_s0): [
    <av - 2, 6>,
    <av, 4>
]
Rational roots in a: [ <0, 4>, <2, 6> ]
Derived (a,b) solutions (when L!=0):
<0, -1, true, true>
<2, "L=0", "C=0 ?", true>

--- Fiber s = 4 ---
Factorization(F_s0): [
    <av^10 + 120*av^9 + 3795*av^8 - 32830*av^7 - 2213145*av^6 + 21454836*av^5 +
        456975685*av^4 - 10459046190*av^3 + 83904838560*av^2 - 309104177760*av +
        369152308224, 1>
]
Rational roots in a: []

--- Fiber s = 4/9 ---
Factorization(F_s0): [
    <av^10 - 1672/81*av^9 + 42505/243*av^8 - 417479710/531441*av^7 +
        87234203815/43046721*av^6 - 1196665550188/387420489*av^5 +
        798522439495909/282429536481*av^4 - 35430471909028750/22876792454961*av\
        ^3 + 101165418048128800/205891132094649*av^2 -
        154140701583484000/1853020188851841*av +
        3719235255680000/617673396283947, 1>
]
Rational roots in a: []

--- Fiber s = 2 ---
Factorization(F_s0): [
    <av^10 + 16*av^9 - 81*av^8 - 1698*av^7 + 5403*av^6 + 68604*av^5 -
        302727*av^4 - 763506*av^3 + 6817596*av^2 - 14294840*av + 9825088, 1>
]
Rational roots in a: []

-------------------------------------------------------------------------------
Step 4. Degenerate locus: solve L(s,a)=0 and C(s,a)=0 via resultant in s
-------------------------------------------------------------------------------

Resultant Res_s(L,C) as polynomial in a has degree: 27
Factorization of Res_s(L,C) in Q[a]:
[
    <aa - 2, 6>,
    <aa^21 + 285345/2048*aa^20 - 52389295807/16777216*aa^19 +
        522650508851/16777216*aa^18 - 12868810960899/67108864*aa^17 +
        55282038601431/67108864*aa^16 - 88360707927419/33554432*aa^15 +
        1747468207486969/268435456*aa^14 - 855371582178461/67108864*aa^13 +
        2687766826029291/134217728*aa^12 - 1696891713504949/67108864*aa^11 +
        6789739499416729/268435456*aa^10 - 1288675394382575/67108864*aa^9 +
        659577032682517/67108864*aa^8 - 3879986488883/2097152*aa^7 -
        15490767375299/8388608*aa^6 + 4080050071429/2097152*aa^5 -
        1864350017295/2097152*aa^4 + 55694069793/262144*aa^3 -
        23895604035/1048576*aa^2 + 179891901/262144*aa + 2099601/262144, 1>
]
Rational a-roots forced by Res_s(L,C)=0 (linear factors only): [ <2, 6> ]

--- Degenerate analysis at a = 2 ---
gcd_s(L,C) = sv^3 - sv^2 - sv + 1
Roots of gcd in s (rational): [ <-1, 1>, <1, 2> ]
Factorization of L(s,a0): [
    <sv - 1, 2>,
    <sv + 1, 1>,
    <sv^3 - 13*sv^2 - 14*sv + 18, 1>
]
Factorization of C(s,a0): [
    <sv - 1, 2>,
    <sv + 1, 1>,
    <sv^4 - 19/4*sv^3 + 15/4*sv^2 + 9*sv - 7, 1>
]

-------------------------------------------------------------------------------
Step 5. Projective curve defined by F(s,a)=0: singularities and genus
-------------------------------------------------------------------------------

Projective curve defined. Degree = 17
Singular points (projective):
{@ (-1 : 2 : 1), (0 : 0 : 1), (1 : 0 : 1), (1 : 2 : 1), (-1 : 1 : 0), (0 : 1 :
0) @}
Arithmetic genus = 120
Geometric genus  = 7
Genus            = 7
Computing rational points on Cproj with bound H = 1000000000

Rational Points found on Cproj:
{@ (1 : 2 : 1), (0 : 1 : 0), (0 : 0 : 1), (0 : 6 : 1), (1 : 0 : 0), (-1 : 2 :
1), (1 : 0 : 1), (-1 : 1 : 0) @}

Interpretation:
Affine solution (s, a) = <1, 2>
Point at infinity: (0 : 1 : 0)
Affine solution (s, a) = <0, 0>
Affine solution (s, a) = <0, 6>
Point at infinity: (1 : 0 : 0)
Affine solution (s, a) = <-1, 2>
Affine solution (s, a) = <1, 0>
Point at infinity: (-1 : 1 : 0)
\end{term}


\end{document}